\documentclass[twoside,11pt,reqno]{amsart}
\usepackage{amsmath,amssymb,amscd,mathrsfs,amscd}
\usepackage{graphics,verbatim}
\usepackage{amsthm}
\usepackage{graphicx}
\usepackage{amsfonts}
\usepackage{amsopn}
\usepackage{amsmath, amsfonts, amssymb}
\usepackage{tikz}

\newtheorem{thm}{Theorem}[subsection]

\newtheorem{lem}{Lemma}[subsection]
\newtheorem{prop}{Proposition}[subsection]
\newtheorem{rk}{Remark}[subsection]

\newtheorem{cor}{Corollary}[subsection]

\newcommand{\Z}{\mathbb{Z}}
\newcommand{\N}{\mathbb{N}}
\newcommand{\Q}{\mathbb{Q}}
\newcommand{\C}{\mathbb{C}}

\newcommand{\p}{P_\bullet}
\newcommand{\x}{X_{\0}^+}

\newcommand{\osp}{\mathfrak{osp}}
\newcommand{\s}{\mathfrak{sl}}
\newcommand{\Ext}{\operatorname{Ext}}
\newcommand{\Hom}{\operatorname{Hom}}
\newcommand{\atyp}{\operatorname{atyp}}
\newcommand{\core}{\operatorname{core}}
\newcommand{\fg}{\mathfrak{g}}
\newcommand{\fh}{\mathfrak{h}}
\newcommand{\fb}{\mathfrak{b}}
\newcommand{\fp}{\mathfrak{p}}
\newcommand{\ff}{\mathfrak{f}}
\newcommand{\gl}{\mathfrak{gl}}
\newcommand{\F}{\mathcal{F}}
\newcommand{\fv}{\mathcal{V}}
\newcommand{\X}{\mathcal{X}}
\newcommand{\0}{\bar 0}
\newcommand{\1}{\bar 1}
\newcommand{\tf}{\tilde{f}}

\def \dd {D(2,1;\al)}
\def \tq {\ |\ \ }

\def \dd {D(2,1;\alpha)}
\def \gmod {\fg\hbox{-}\mathrm{mod}}

\let \l=\lambda
\let \Ga=\Gamma
\let \a=\alpha

\let \d=\delta
\let \Ga=\Gamma
\let \l=\lambda

\let \eps=\varepsilon
\let \w=\omega
\let\e=\eps

\numberwithin{equation}{subsection}

\begin{document}

\title{Complexity of modules over classical Lie superalgebras}
\author{Houssein El Turkey}

\address{Department of Mathematics and Physics, University of New Haven, 300 Boston Post Road, West Haven, CT 06516, United States}
\email{helturkey@newhaven.edu}
\date{\today}
\maketitle
\begin{abstract}
The complexity of the simple and the Kac modules over the general linear Lie superalgebra $\mathfrak{gl}(m|n)$ of type $A$ was computed by Boe, Kujawa, and Nakano in \cite{BKN1}. A natural continuation to their work is computing the complexity of the same family of modules over the ortho-symplectic Lie superalgebra $\osp(2|2n)$ of type $C$. The two Lie superalgebras are both of \emph{Type I} which will result in similar computations. In fact, our geometric interpretation of the complexity agrees with theirs. We also compute a categorical invariant, $z$-complexity, introduced in \cite{BKN1}, and we interpret this invariant geometrically in terms of a specific detecting subsuperalgebra. In addition, we compute the complexity and the $z$-complexity of the simple modules over the \emph{Type II} Lie superalgebras $\osp(3|2)$, $\dd$, $G(3)$, and $F(4)$.
\end{abstract}

\section{Introduction}

Let $\fg=\fg_{\0}\oplus\fg_{\1}$ be a classical Lie superalgebra (hence $\fg_{\0}$ is a reductive Lie algebra) over the complex numbers, $\C$. Let $\F:=\F_{(\fg,\fg_{\0})}$ be the category of finite-dimensional $\fg$-supermodules which are completely reducible over $\fg_{\0}$. The authors in \cite{BKN2} showed that $\F$ has enough projectives and it satisfies: $(i)$ it is a self-injective category and $(ii)$ every module in this category admits a projective resolution which has a polynomial rate of growth. For a module $M\in \F$, the complexity $c_{\F}(M)$ is the rate of growth of the minimal projective resolution of $M$. 

In this paper we compute the complexity of the simple and the Kac modules for the orthosymplectic Lie superalgebra $\osp(2|2n)$. Let $K(\lambda)$ (resp. $L(\lambda)$) be the Kac (resp. simple) module of highest weight $\lambda$. Let $\atyp(\lambda)$ denote the atypicality of $\lambda$ (see Subsection~\ref{ss:atyp}). For $\osp(2|2n)$, $\atyp(\lambda)$ is either zero or one. For typical $\lambda$ (i.e. $\atyp(\lambda)=0$), the simple and the Kac modules are projective and hence they have a zero complexity. For atypical $\lambda$ (i.e. $\atyp(\lambda)=1$), the complexity is computed in Theorems~\ref{t2:atypcomplexitysimple} and \ref{t1:atypcomplexityKac}:
\[c_{\F}(L(\lambda))=2n+1,\quad c_{\F}(K(\lambda))=2n.\]

These computations can be interpreted geometrically as follows. For a module $M$, let $\X_M$ denote the associated variety defined by Duflo and Serganova \cite {DS}, and $\fv_{(\fg,\fg_{\0})}(M)$ the support variety as defined in \cite{BKN4}. Then, if $X(\lambda)$ is a Kac or a simple module, we have the geometric interpretation of the complexity in Theorem~\ref{ospgeometric1}:
\begin{equation}\label{compinterpret}
 c_{\F}(X(\lambda))=\dim \X_{X(\lambda)}+\dim \fv_{(\fg,\fg_{\0})}(X(\lambda)).
\end{equation}

The authors in \cite{BKN1} introduced a categorical invariant called the $z$-complexity of modules and denoted it by $z_{\F}(-)$ (see \cite[Section~9]{BKN1}). They computed the $z$-complexity of the simple and the Kac modules over $\gl(m|n)$ and then used a detecting subsuperalgebra $\ff$ to interpret their computations geometrically. We carry these computations over to $\osp(2|2n)$ and conclude in Theorems~\ref{t2:zcompsimple} and \ref{t:zcompkac} that for an atypical $\l$, we have
\begin{equation}
z_{\F}(L(\lambda))=2, \quad z_{\F}(K(\lambda))=1.
\end{equation}
Moreover, we show in Theorem~\ref{ospgeometric2} that if $X(\lambda)$ is a Kac or a simple module, we have
\begin{equation}\label{zcompinterpret}
z_{\F}(X(\lambda))=\dim \fv_{(\ff,\ff_{\0})}(X(\lambda)).
\end{equation}

The fact that our geometric interpretations of the complexity and the $z$-complexity agree with the results obtained in \cite{BKN1} was expected since both types $A$ and $C$ are \emph{Type I} Lie superalgebras (Subsection~\ref{ospprelim}). It was interesting to know if these interpretations would hold over \emph{Type II} Lie superalgebras, hence we computed the complexity and the $z$-complexity of the simple (finite-dimensional) modules over $\osp(3|2)$, and the three exceptional Lie superalgebras $\dd$, $G(3)$, and $F(4)$. Our results show that equations~\eqref{compinterpret} and \eqref{zcompinterpret} hold for the simple modules over these Lie superalgebras. The results in this paper raise the question of whether these geometric interpretations will hold over other classical Lie superalgebras, in particular types $B$ and $D$.

The paper is organized as follows. In Section~\ref{s:prel}, we introduce the preliminaries for classical Lie superalgebras and their representations. We recall the definitions of atypicality, complexity, support variety, associated variety, and $z$-complexity of modules. In Section~\ref{s:computecomplexity}, we compute the complexity of simple and Kac modules over $\osp(2|2n)$. We construct an explicit minimal projective resolution of the trivial module and use the fact that simple modules of the same atypicality have the same complexity (\cite[Theorem~4.1.1]{kujawa}). We then use the equivalence of blocks defined in \cite{GS} to compute the complexity of the Kac modules. In Section~4, we relate the support and associated variety of these modules to complexity as mentioned earlier. In Section~5, we compute the $z$-complexity of simple and Kac modules over $\osp(2|2n)$ and show that $z$-complexity will be the same as the dimension of the support variety over a detecting subalgebra, $\ff$, of $\osp(2|2n)$. In Section~\ref{s:additionalexamples} we compute the complexity and the $z$-complexity of the simple modules over $\osp(3|2),\,\dd,\,G(3),$ and $F(4)$. We also show that the same geometric interpretations hold in these cases.\\

\subsection{Acknowledgments} 
I would like to acknowledge my Ph.D. adviser, Jonathan Kujawa, for many insightful discussions. His comments were very helpful throughout this work. I would also like to thank the referee(s) for the useful remarks and valuable input. 

\section{Preliminaries}\label{s:prel}
\subsection{Lie superalgebras and representations}
We will use the notations and conventions developed in \cite{BKN1,BKN2,BKN3,BKN4}. We will work over the complex numbers $\C$ throughout this paper.\\

Let $\fg=\fg_{\0}\oplus\fg_{\1}$ be a Lie superalgebra with a bracket operation $[\,,\,]:\fg\otimes\fg\rightarrow \fg$ which preserves the $\Z_2$-grading and satisfies graded versions of the usual Lie bracket axioms. The subspace $\fg_{\0}$ is a Lie algebra under the bracket and $\fg_{\1}$ is a $\fg_{\0}$-module. Elements of $\fg_{\0}$ (resp. $\fg_{\1}$) are called \emph{even} (resp. \emph{odd}). A finite-dimensional Lie superalgebra $\fg$ is called \emph{classical} if there is a connected reductive algebraic group $G_{\0}$ such that $\operatorname{Lie}(G_{\0})=\fg_{\0}$, and an action of $G_{\0}$ on $\fg_{\0}$ which differentiates to the adjoint action of $\fg_{\0}$ on $\fg_{\1}$. 
 A \emph{basic classical} Lie superalgebra is a classical Lie superalgebra with a nondegenerate invariant supersymmetric even bilinear form. The Lie superalgebras considered in this paper are basic classical Lie superalgebras.

We refer the reader to \cite{BKN1} for the definition of the category of $\fg$-supermodules. We write $\F=\F_{(\fg,\fg_{\0})}$ for the full subcategory of all finite-dimensional $\fg$-supermodules which are completely reducible over $\fg_{\0}$. As only supermodules will be considered in this paper, we will from now on use the term ``module" with the understanding that the prefix ``super'' is implicit. 

\subsection{Atypicality}\cite[Section~7.2]{BKN4}\label{ss:atyp}
Let $\fg$ be a basic classical Lie superalgebra with a (dual) Cartan subalgebra $\fh$ $(\fh^*)$. Then $\fh^*$ is equipped with a bilinear form $(,)$. Positive (negative) even and odd roots are defined in \cite{Kac}. Let $\rho$ be half the sum of the positive even roots minus half the sum of the positive odd roots. For $\l\in \fh^*$, define the \emph{atypicality} of $\lambda$ to be the maximal number of pairwise orthogonal positive isotropic roots which are also orthogonal to $\lambda+\rho$ with respect to the form $(,)$. We will write $\atyp(\lambda)$ for the atypicality of $\lambda$. In the Lie superalgebras considered in this paper, the atypicality of a weight is either zero (typical) or one (atypical). If $\gamma$ is the positive isotropic root orthogonal to $\lambda+\rho$, we say $\lambda$ is atypical with respect to $\gamma$.

Define the atypicality of a simple $\fg$-module $L(\lambda)$ of highest weight $\lambda$ by $\atyp(L(\lambda)):=\atyp(\lambda)$. It is known that the atypicality of a simple module is the same for all simple modules in a given block. Hence it makes sense to refer to the atypicality of a block.  
\subsection{Complexity}\cite[Section~2.2]{BKN1}
Let $\{V_t\,|\, t \in \N\} = \{V_\bullet\}$ be a sequence of finite-dimensional $\C$-vector spaces. The rate of growth of $V_\bullet$, $r(V_\bullet)$, is the smallest nonnegative integer c such that there exists a constant $C > 0$ with $\dim V_t \leq C\cdot t^{c-1}$ for all $t$. If no such integer exists
then $V_\bullet$ is said to have infinite rate of growth.

Let $M\in \F$ and $\p\twoheadrightarrow M$ be a minimal projective resolution. Define the \emph{complexity} of $M$ to be $c_{\F}(M):=r(\p)$. As shown in \cite[Theorem~2.5.1]{BKN2} the complexity is always finite, in particular if $M$ is an object of $\F$, then $c_{\F}(M)\leq \dim \fg_{\1}$. Moreover, \cite[Proposition~2.8.1]{BKN2} provides a characterization of the complexity via rates of growth of extension groups in $\F$:
\[c_{\F}(M)=r\Big(\Ext_{(\fg,\fg_{\0})}^\bullet(M,\bigoplus S^{\dim P(S)})\Big),\]
where the sum is over all the simple modules $S\in \F$, and $P(S)$ is the projective cover of $S$. Here and elsewhere, we write $\Ext_{(\fg,\fg_{\0})}^\bullet(M,N)$ for the relative cohomology for the pair $(\fg,\fg_{\0})$ as introduced in \cite[Section~2.3]{BKN3}. In some sense, the complexity of a module measures how far the module is from being projective. For example, by \cite[Corollary~2.7.1]{BKN2}, $c_{\F}(M)=0$ if and only if $M$ is projective.
 
\subsection{Support variety}\cite[Section~6]{BKN3}\label{suppvar}
 Let $R=H^{\bullet}(\fg,\fg_{\0};\C)$ be the cohomology ring of $\fg$ and let $M\in \F$. According to \cite[Theorem~2.7]{BKN3}, $\Ext_{\F}^{\bullet}(M,M)$ is a finitely generated $R$-module. Set $J:=Ann_R(\Ext_{\F}^{\bullet}(M,M))$. The \emph{support variety} of $M$ is defined by
 \[\fv_{(\fg,\fg_{\0})}(M):=MaxSpec(R/J).\]

 \subsection{Associated variety}\cite[Section~2]{DS}\label{assocvar}
 Let $\X=\{x\in \fg_{\1} \mid [x,x]=0\}$. If $M \in \F$, then Duflo and Serganova \cite{DS} defined an \emph{associated variety} of $M$ which is equivalent to:
\[\X_M=\{x\in \X \mid \text{$M$ is not projective as a $U(\langle x \rangle)$-module}\}\cup \{0\},\]
where $U(\langle x \rangle)$ denotes the enveloping algebra of the Lie superalgebra generated by $x$.

\subsection{$z$-complexity}(\cite[Section~9]{BKN1})\label{zcompdef}
Let $M\in \F$. The $z$-\emph{complexity} of $M$ is
\[z_{\F}(M):=r\Big(\Ext_{(\fg,\fg_{\0})}^\bullet(M,\bigoplus S)\Big),\]
where the direct sum runs over all simple modules $S\in\F$. Unlike complexity, $z_{\F}(-)$ has the advantage of being invariant under category equivalences.

If $\p\twoheadrightarrow M$ is a minimal projective resolution of $M$, define $s(\p)$ to be the rate of growth of the number of indecomposable summands at each step in the resolution. We can easily show that $z_{\F}(M)=s(\p)$.

\section{Computing the complexity for $\osp(2|2n)$}\label{s:computecomplexity}

 \subsection{The Lie superalgebra $\osp(2|2n)$}\label{ospprelim}

 Consider the matrix realization of $\osp(2|2n)$ given in \cite{Kac}:
\[
\fg=\left[\begin{array}{c|c}
\begin{array}{cc}
\alpha & 0\\
0 & -\alpha\\
\end{array}  & \begin{array}{cc}
x_1 & x_2\\
y_2 & y_1\\
\end{array}\\ \hline
\begin{array}{cc}
y_1^t & x_2^t\\
-y_2^t & -x_1^t\\
\end{array} & \begin{array}{cc}
a & b\\
c & -a^t\\
\end{array}\\
\end{array}
\right],
\]
where $x_1,\,x_2,\,y_1,\,y_2$ are $1\times n$ matrices; $a,\,b,\,c$ are $n\times n$ matrices with $b$ and $c$ being symmetric; $\alpha$ is a scalar in $\C$. The diagonal blocks form the even part $\fg_{\0}\cong \C\oplus \mathfrak{sp}(2n)$ while the anti-diagonal blocks form the odd part $\fg_{\1}$. The Lie super-bracket is defined by
 \[[A,B]=AB-(-1)^{\bar{A}\bar{B}}BA,\]
 for homogeneous elements $A,\,B\in \fg_{\0}$ or $\fg_{\1}$. We then extend the definition of the bracket to all of $\fg$ by bilinearity.

The Cartan subalgebra $\fh$ is chosen to be the set of diagonal matrices in $\fg$. Let $\varepsilon_1:\fh\rightarrow \C$ be the linear map that takes an element of $\fh$ to its first diagonal entry. For $1\leq i \leq n$, let $\delta_i:\fh\rightarrow \C$ be the linear map that takes an element of $\fh$ to the $i$th diagonal entry of the matrix in the second diagonal block. The set $\{\varepsilon_1,\,\delta_1,\,\delta_2,\ldots,\delta_n \}$ forms a basis of $\fh^*$ which is endowed with a nondegenerate symmetric bilinear form $(\,,\,)$ given by
\[
(\varepsilon_1,\varepsilon_1)=1,\quad (\delta_i,\delta_j)=-\delta_{ij},\quad (\varepsilon_1,\delta_i)=0,
\]
for all $1\leq i,j \leq n$. The set of simple roots is
\[
\Delta=\{\delta_i-\delta_{i+1}\mid 1\leq i\leq n\}\cup\{2\delta_n,\,\varepsilon_1-\delta_1\}.
\]
The even roots of $\fg$ are:
\[
\Phi_{\0}=\{\pm\delta_i\pm\delta_j\mid 1\leq i\neq j\leq n\}\cup\{\pm2\delta_i\mid 1\leq i \leq n\},
\]
while the odd roots are
\[
\Phi_{\1}=\{\pm\varepsilon_1\pm\delta_i\mid 1\leq i \leq n\}.
\]
The positive roots of $\fg$ are
\[
\Phi^{+}=\{\delta_i+\delta_j\mid 1\leq i\leq j\leq n\}\cup\{\delta_i-\delta_j\mid 1\leq i<j \leq n\}\cup\{\varepsilon_1\pm\delta_i\mid 1\leq i \leq n\}.
\]

Note that $\fg$ has a $\Z$-grading given by $\fg=\fg_{-1}\oplus\fg_{0}\oplus\fg_{1}$ where $\fg_{0}$ is the subalgebra spanned by the even root vectors and $\fg_{1}$ (resp. $\fg_{-1}$) is the subalgebra spanned by the odd positive (resp. negative) root vectors. Thus $\fg$ is a \emph{Type I} Lie superalgebra. Note that $\fg_{-1}$ and $\fg_1$ are abelian, $\fg_{\1}=\fg_{-1}\oplus \fg_{1}$, and $\fg_{\0}=\fg_0$. Let $\fb$ be the Borel subalgebra of $\fg$ spanned by the positive root vectors, then we have $\fb=\fb_{\0}\oplus\fg_1$ where $\fb_{\0}$ is spanned by the even positive roots.

The simple modules over $\fg=\osp(2|2n)$ can be constructed as follows. Let $\x\subseteq\fh^*$ be the parameterizing set of highest weights for the simple finite-dimensional $\fg_{\0}$-modules with respect to the pair $(\fh,\,\fb_{\0})$. An explicit description of $\x$ is
\[\x=\{\lambda=\lambda_{-1}\varepsilon_1+\sum_{i=1}^{n}\lambda_i\delta_i \mid \lambda_{-1}\in \C,\,\lambda_i\in \Z,\, \forall i\geq 1;\,\,\lambda_1\geq \lambda_2\geq \ldots \geq \lambda_n\geq 0\}.\]


For $\lambda\in \x$, let $L_{\0}(\lambda)$ be the simple finite-dimensional $\fg_{\0}$-module of highest weight $\lambda$. Since $\osp(2|2n)_{\0}\cong \C\oplus \mathfrak{sp}(2n)$, the simple $\fg_{\0}$-modules are of the form
$L_{\0}(\lambda)=\C_{\lambda_{-1}}\boxtimes L_0(\lambda)$, where $L_0(\lambda)$ is the simple $\mathfrak{sp}(2n)$-module of weight $\sum_{i=1}^{n}\lambda_i\delta_i$. Note the slight abuse of notation where we used $\l$ for the simple $\mathfrak{sp}(2n)$-module while we removed the $\l_{-1}$-part.

Set $\fp^{+}=\fg_{\0}\oplus\fg_{1}$. Since $\fg_{1}$ is an abelian ideal of $\fp^{+}$, view $L_{\0}(\lambda)$ as a simple finite-dimensional $\fp^{+}$-module via inflation. For $\lambda \in X_0^+$, the \emph{Kac module} $K(\lambda)$ is defined by:
\[K(\lambda):=U(g)\otimes_{U(\fp^+)}L_{\0}(\lambda).\]

The Kac module $K(\lambda)$ has a unique maximal submodule. The head of $K(\lambda)$ is the simple finite-dimensional $\fg$-module $L(\lambda)$. The set $\{L(\lambda)\mid \lambda\in X_0^+\}$ is a complete set of non-isomorphic simple modules in $\F=\F_{(\fg,\fg_{\0})}$. 

From now on, we denote by $P(\l)$ the projective cover of $L(\l)$. By \cite[Theorem~1]{Kac1} we know that if $\atyp(\lambda)=0$, then $P(\lambda)=L(\lambda)=K(\lambda)$ and hence:
\begin{prop}\label{typcomplexity}
	If $\lambda$ is typical, then $c_{\F}(L(\lambda))=c_{\F}(K(\lambda))=0$.
\end{prop}

\begin{rk}
In $\osp(2|2n)$, we have $\rho=-n\varepsilon_1+\sum_{i=1}^{n}(n-i+1)\delta_i$. Using $\rho$ we can show that for any $\l\in \x$, $\atyp(\lambda)$ is either zero or one. Note that if $\l=\lambda_{-1}\varepsilon_1+\sum_{i=1}^{n}\lambda_i\delta_i=(\l_{-1}|\l_1,\ldots,\l_n)$ is atypical, then $\l_{-1}\in  \Z$. We also note that the zero weight corresponding to the trivial module $L(0)=\C$ has $\atyp(0)=1$. 
 \end{rk}

\subsection{Complexity of simple modules for $\fg=\osp(2|2n)$}\label{SS:simplecomplexity}

 For a typical $\l$, the simple module $L(\l)$ is projective and hence has zero complexity. We only need to consider the case when $\lambda$ is atypical. In this section, we will refer to \cite[Sections~1.2,\,3.1]{CWZ} to get a description of the projective covers.

Let $W$ be the Weyl group of $\fg$ which, by definition, is generated by the reflections corresponding to the even simple roots of $\fg$. If $\lambda\in \x$ is atypical with respect to the odd positive root $\gamma$, the authors in \cite{CWZ} defined an ``$L$-operator" given by:
\[\lambda^L:=\omega(\lambda+\rho-k\gamma)-\rho,\]
where $k$ is the smallest positive integer such that $(\lambda+\rho-k\gamma,\alpha_i)\neq 0$ for all the even simple roots $\alpha_i$ and $\omega$ is the unique element in the Weyl group of $\mathfrak{sp}(2n)$ rendering $\lambda^L$ dominant. Given an atypical $\lambda\in \x$, we shall write

\begin{equation}\label{eq1:inductivepower}
\lambda^{(0)}=\lambda,\quad \lambda^{(l+1)}=(\lambda^{(l)})^L,\quad l\geq 0.
\end{equation}

Let $\omega_0$ be the longest element in the Weyl group of $\mathfrak{sp}(2n)$ and let $\beta=2n\varepsilon_1$ be the sum of all the odd positive roots. We can use \cite[equation~(4)]{CWZ} to get:
\[\lambda=\omega_0^{-1}\Big(\beta-(\beta-\omega_0\lambda^L)^L\Big),\]
which proves that the $L$-operator is one-to-one. We will prove that $L$ is onto on the set of atypical weights in $\x$ using a representation theoretical approach (Proposition~\ref{Lsurjective}). Using this we can define $\lambda^{-L}$ to be the unique weight $\mu$ such that $\mu^L=\lambda$, hence we can extend the definition in equation~\eqref{eq1:inductivepower} to any $l\in \Z$.

To compute $\lambda^L$ for each atypical $\lambda$, we need the $(n+1)$-tuple $f_\lambda:=((f_\lambda)_{-1}|(f_\lambda)_{1},\cdots,(f_\lambda)_{n})$ of integers introduced in \cite[Section~1.2]{CWZ}:
\[(f_\lambda)_{-1}=(\lambda+\rho,\varepsilon_1),\quad (f_\lambda)_i=(\lambda+\rho,\delta_i),\,i\geq 1.\]
 Since $\l$ is atypical, then  $|(f_\l)_{-1}|=-(f_\l)_i$ for some $i\geq 1$. Set $f_\l^L:=f_{\lambda^L}$. Using the description of $f_\l^L$ given in \cite[Section~1.2]{CWZ}, the following computations can be performed.

\begin{lem}\label{t1:projcover(2,2n)}
We have:
\begin{enumerate}
\item For $d\geq 0$, $(-d|d,0,\ldots,0)^L=(-d-1|d+1,0,\dots,0)$.
\item For $d\geq 1$, $(2n+d|d,0,\ldots,0)^L=(2n+(d-1)|d-1,0,\ldots,0)$.
\item  $(2n|0,\dots,0)^L=(0|0,\ldots,0)$.
\item For $\l=(0|0,\ldots,0)=0$ and $d\geq 0$ we have
\[\l^{(d)}=(-d|d,0,\ldots,0)\,\, \text{and}\,\, \l^{(-d-1)}=(2n+d|d,0,\ldots,0).\]
\end{enumerate}
\end{lem}

 From \cite[Theorem~7]{CWZ}, we have the following 2-step Kac flag:
 \begin{equation}\label{KacFlag}
0\rightarrow K(\lambda)\rightarrow P(\lambda^L)\rightarrow K(\lambda^L)\rightarrow 0,
\end{equation}
which implies that $\dim P(\lambda^L)=\dim K(\lambda)+\dim K(\lambda^L)$. Moreover, using \cite[Corollary~8]{CWZ}, there is a short exact sequence:
\[0\rightarrow L(\lambda^L)\rightarrow K(\lambda)\rightarrow L(\lambda) \rightarrow 0,\]
which shows that the Kac module has two layers: the head of $K(\lambda)$ is $L(\lambda)$ and the socle is $L(\lambda^L)$. Note that by the Kac filtration of $P(\l^L)$ and the composition factors of $K(\l)$ we know that the socle of $P(\l^L)$ has only one submodule, namely $L(\l^L)$. Similarly the head of $P(\l^L)$ is the head of $K(\l^L)$ which is $L(\l^L)$. Using the description of Kazhdan-Lusztig polynomials in \cite[Theorem~5, Remark~1]{CWZ}, we can show  that
\[\dim \Ext^1(L(\l),L(\mu))=1\Leftrightarrow \l=\mu^L\,\,\text{or}\,\, \l^L=\mu.\]
Thus there is no indecomposable module $M$ such that the following sequence is exact:
 \[0\rightarrow L(\lambda)\rightarrow M\rightarrow L(\lambda^{LL}) \rightarrow 0.\]
Therefore, the projective module $P(\lambda^L)$ has the following layer structure:

\[\begin{tikzpicture}

  \node (n1) at (1,1) {$L(\l^L)$};
  \node (n2) at (0,0)  {$L(\l)$};
  \node (n3) at (2,0)  {$L(\l^{LL})$};
  \node (n4) at (1,-1) {$L(\l^L)$};

  \foreach \from/\to in {n1/n2,n1/n3,n2/n4,n3/n4}
    \draw (\from) -- (\to);

\end{tikzpicture}\]

To compute the complexity, we need the following bounds on the dimension of the simple $\mathfrak{sp}(2n)$-module $L_0(r,0,\ldots,0)$.  
\begin{lem}\label{lem:bddimsimple}
If $r=0$, $\dim L_0(r,0,\ldots,0)=1$. If $r\geq 1$, there are positive constants $C$ and $C'$ that depend only on $n$ such that
\[Cr^{2n-1}\leq \dim L_0(r,0,\ldots,0)\leq C'r^{2n-1}.\]
\end{lem}

\begin{proof}
Let $r\geq 1$. Let $\delta$ be half the sum of the positive roots in $\mathfrak{sp}(2n)$, then $\delta=\displaystyle{\sum_{i=1}^{n}(n-i+1)\delta_i}$. By the Weyl-dimension formula (\cite[Section~24.3]{HUM}) we have
\[\dim L_0(r,0,\ldots,0)=\frac{(2n+2r)\prod_{j=2}^n(r+j-1)(2n+r-j+1)}{(2n)\prod_{j=2}^n(j-1)(2n-j+2)}.\]
Let
\[
C=\frac{1}{(2n)\prod_{j=2}^n (j-1)(2n-j+2)}.
\]
 Then $C$ is a positive constant depending only on $n$ and $\dim L_0(r,0,\ldots,0)$ is a polynomial in $r$ of degree $2n-1$ with a positive leading coefficient. Moreover,
\begin{align*}
\dim L_0(r,0,\ldots,0)&=C(2n+2r)\prod_{j=2}^n(r+j-1)(2n+r-j+1)\\
                      &\geq C(2n+2r)(r+1)^{n-1}(n+r+1)^{n-1} \geq Cr^{2n-1}.
\end{align*}
On the other hand,
\begin{align*}
\dim L_0(r,0,\ldots,0)&=C(2n+2r)\prod_{j=2}^n(r+j-1)(2n+r-j+1)\\
                      &\leq C(2n+2r)(r+n-1)^{n-1}(2n+r-1)^{n-1}.\\
\end{align*}
Let us pick positive constants $C_2,C_3,C_4$ depending only on $n$ such that
\[2n+2r\leq C_2r,\quad r+n-1\leq C_3r,\quad 2n+r-1\leq C_4r,\]
then
\[\dim L_0(r,0,\ldots,0)\leq C \cdot C_2\cdot (C_3)^{n-1}(C_4)^{n-1}r^{2n-1}=C'r^{2n-1},\]
where $C'$ is a positive constant that depends only on $n$.
\end{proof}

\begin{thm}\label{t2:atypcomplexitysimple}
For atypical $\lambda\in \x$, $c_{\F}(L(\lambda))=2n+1$.
\end{thm}

\begin{proof}
First, we find the complexity of the trivial module $\C=L(0|0,\ldots,0)$. For $\l\in \x$, the definition of $\l^{(i)}$ was given in equation~\eqref{eq1:inductivepower}. For $0=(0|0,\ldots,0)$ and $i\in \Z$ we have by Lemma~\ref{t1:projcover(2,2n)}:
\[0^{(i)}=(-i|i,0,\ldots,0)\,\, \text{and}\,\, 0^{(-i-1)}=(2n+i|i,0,\ldots,0).\]
For $i\in \Z$, set
$$[i]=L(0^{(i)}),\quad P(i)=P(0^{(i)}),\,\,\text{and}\,\, K(i)=K(0^{(i)}).$$
 Using these notations, the projective cover $P(i)$ has the following radical layer structure:
\[\begin{tikzpicture}

  \node (n1) at (1,1) {$[i]$};
  \node (n2) at (0,0)  {$[i-1]$};
  \node (n3) at (2,0)  {$[i+1]$};
  \node (n4) at (1,-1) {$[i]$};

  \foreach \from/\to in {n1/n2,n1/n3,n2/n4,n3/n4}
    \draw (\from) -- (\to);

\end{tikzpicture}\]
The minimal projective resolution of $L(0)$ is
\begin{equation}\label{projrestrivial}
 \cdots P_d \stackrel{f_d}{\rightarrow} \cdots \stackrel{f_1}{\rightarrow}P_0= P(0)\stackrel{f_0}{\rightarrow} L(0) \rightarrow 0,
\end{equation}
where the $d$th term in this resolution is given as follows. If $d$ is even,
\begin{align*}
P_d=P(-d)\oplus P(-d+2)\oplus\cdots \oplus P(-2)\oplus P(0)\oplus P(2)\oplus \cdots \oplus P(d-2)\oplus P(d),
\end{align*}
where the kernel of $f_d$ is
\[Ker(f_d)=\raisebox{-3em}{\begin{tikzpicture}
	
	\node (n1) at (1,1) {$[-d-1]$};
	\node (n2) at (2,0)  {$[-d]$};
	\node (n3) at (3,1)  {$[-d+1]$};
	\node (n4) at (4,0)  {$\cdots$};
	\node (n5) at (5,1) {$[-1]$};
	\node (n6) at (6,0)  {$[0]$};
	\node (n7) at (7,1)  {$[1]$};
	\node (n8) at (8,0)  {$\cdots$};
	\node (n9) at (9,1) {$[d-1]$};
	\node (n10) at (10,0)  {$[d]$};
	\node (n11) at (11,1)  {$[d+1]$};
	
	\foreach \from/\to in {n1/n2, n2/n3, n3/n4, n4/n5, n5/n6, n6/n7, n7/n8, n8/n9, n9/n10, n10/n11}
	\draw (\from) -- (\to);
	
	\end{tikzpicture}}\]

If $d$ is odd, $P(\pm 2)$ are replaced by $P(\pm 1)$, and $P(0)$ is removed from $P_d$. The kernel of $f_d$ is described similarly. An inductive argument proves this result using the diagrammatic method for modular representations given in \cite{BC}. In particular, we use the description given in \cite[Section~10.3]{BC} for the kernel of the surjective map $P(M) \twoheadrightarrow M$ where $P(M)$ is the projective cover of $M\in \F$. 

We have $\dim P(\lambda^L)=\dim K(\lambda^L)+\dim K(\lambda)$. Then, for $i\in \Z$,
\begin{equation}\label{dimprojkac}
\dim P(i)=\dim K(i)+\dim K(i-1).
\end{equation}
By the PBW basis of $U(\fg)$, we have
\begin{equation}\label{dimkac}
\dim K(\lambda)=2^{\dim \fg_{-1}}\cdot \dim L_{0}(\lambda),
 \end{equation}
 where $L_{0}(\lambda)$ is the simple $\mathfrak{sp}(2n)$-module of highest weight $\lambda$ (where we omit the $\lambda_{-1}$ from $\lambda)$. By equations~\eqref{projrestrivial}, \eqref{dimprojkac}, \eqref{dimkac}, and by Lemma~\ref{lem:bddimsimple} we can show that
\begin{equation}\label{eq2:dimPd}
kd^{2n}\leq \dim P_d \leq k'd^{2n}\quad\text{for all $d\geq 1$,}
\end{equation}
where $k,\,k'$ are positive constants that depend only on $n$. Thus, $c_{\F}(L(0|0,\ldots,0))=2n+1$. By \cite[Theorem~4.1.1]{kujawa}, all simple modules of the same atypicality have the same complexity. Thus the complexity of all atypical simple $\osp(2|2n)$-modules is $2n+1$.
\end{proof}

\subsection{Complexity of $K(0|0,\ldots,0)$ for $\fg=\osp(2|2n)$ using projective resolutions}
Using the computations done in the above theorem, we compute the complexity of the Kac module $K(0|0,\ldots,0)$.
\begin{prop}\label{comptrivialKac} $c_{\F}(K(0|0,\ldots,0))=2n$.
\end{prop}
\begin{proof}
Using the same notation as in Theorem~\ref{t2:atypcomplexitysimple}, the minimal projective resolution of $K(0)$ is given by:
\begin{equation}\label{projresK(0)}
\cdots \rightarrow P(-2)\rightarrow P(-1)\rightarrow P(0)\rightarrow K(0) \rightarrow 0.
\end{equation}
By equations~\eqref{dimprojkac}, \eqref{dimkac}, and by Lemma~\ref{lem:bddimsimple} we can show that
\begin{equation*}
kd^{2n-1}\leq \dim P_d \leq k'd^{2n-1}\quad\text{for all $d\geq 1$,}
\end{equation*}
where $k,\,k'$ are positive constants that depend only on $n$. Thus, $c_{\F}(K(0|0,\ldots,0))=2n$.
\end{proof}

\subsection{Complexity of Kac modules}\label{compKac}
Recall that if $\l$ is typical, then $K(\l)$ is projective and hence has zero complexity. We will use the complexity of $K(0|0,\ldots,0)$ to compute the complexity of any atypical Kac module. First we show that the $L$-operator on the weights does not change the complexity of the Kac modules. Then we give an explicit description of the principal block $\F^{\chi_{0}}$, where $\chi_{0}$ is the central character corresponding to the weight $\lambda=(0|0,\ldots,0)$. This description will be obtained using the characterization of blocks and the notion of weight diagrams given in \cite[Sections~5,\,6]{GS}. We then use the fact that translation functors preserve the complexity to show that $c_{\F}(K(\lambda))=2n$.

\begin{lem}\label{L1:atypcomplexityKac}
For $\lambda\in \x$, $c_{\F}(K(\lambda))=c_{\F}(K(\lambda^{(l)}))$, $l\in \Z$.
\end{lem}

\begin{proof}
It is sufficient to prove that $c_{\F}(K(\lambda))=c_{\F}(K(\lambda^{L}))$. The complexity of any module $M\in \F$ is given by:
\begin{equation}\label{eq1:complexityExt}
c_{\F}(M)=r(\Ext^{\bullet}_{(\fg,\fg_{\0})}(M,\bigoplus L(\mu)^{\dim P(\mu)})),
\end{equation}
where the sum is over all simple modules in $\F$.  Let $S=\bigoplus L(\mu)^{\dim P(\mu)}$. By applying the functor $\Hom(\,\,\,,S)$ to the following 2-step Kac flag
\[0\rightarrow K(\lambda)\rightarrow P(\lambda^L)\rightarrow K(\lambda^L)\rightarrow 0,\]
 we get a long exact sequence in cohomology. Since $P(\lambda^L)$ is projective, then $\Ext^d(P(\lambda^L),S)$ vanishes for all $d\geq 1$, which gives
\[\Ext^d(K(\lambda^L),S)=\Ext^{d+1}(K(\lambda),S).\]
This implies that $\Ext^{\bullet}_{(\fg,\fg_{\0})}(K(\lambda),S)$ and $\Ext^{\bullet}_{(\fg,\fg_{\0})}(K(\lambda^L),S)$ will have the same rate of growth, hence the theorem follows.
\end{proof}

\subsection{Weight diagrams and translation functors}\label{ss:diagramsfunctors}
  In this subsection we refer the reader to \cite[Section~6]{GS} for the definition of weight diagrams. Assume $\l$ is a dominant atypical weight and let $\lambda+\rho=a_1\varepsilon_1+b_1\delta_1+\ldots+b_n\delta_n$. The weight diagram of $\l$ is a function $\tf_\lambda$ on $\Z_{\geq 0}$ that takes the values $ \{>,\,<,\,\times,\,0\}$ represented by a diagram according to the following algorithm:
\begin{itemize}
\item Put the symbol $>$ in position $t$ if $t=|a_1|$.
\item Put the symbol $<$ in positions $t$ if $t=b_i$ for some $i$.
\item If there are both $>$ and $<$ in the same position replace them by the symbol $\times$. The symbol $\times$ will only appear once since $atyp(\l)=1$.
\item Put $0$ otherwise.
\end{itemize}
As an example, the weight diagram of $0=(0|0,\ldots,0)$ is given by:
\[0,<,<,\ldots,<,\times,\ldots,\]
where the first $0$ is at position $0$, the first $<$ is at position $1$, the last $<$ is at position $n-1$, $\times$ is at position $n$, the dots after that stand for empty positions (or zeros). 

It is important to note the relation with the notation developed in Subsection~3.2 to compute $\lambda^L$. In fact, $(f_\l)_i=(\lambda+\rho,\delta_i)=-b_i$ for $i>0$. Suppose $|a_1|=b_i$ for some $i>0$, define the \emph{core} of $\lambda$ to be the set $\core(\l)=\{b_j\,|\,j\neq i\}$. We will represent $\core(\l)$ by the same diagram as $\tf_\lambda$ but replacing the $\times$ by $0$. For simplicity, we denote the diagram of $\core(\l)$ by $\tf_{\core(\l)}$ even though $\core(\l)$ is not a weight. Thus $\tf_\lambda$ and  $\tf_{\core(\l)}$ have $<,>$ at the same positions. For example, $\tf_{\core(0)}$ is given by:
\[0,<,<,\ldots,<,0,\ldots.\]

As discussed in \cite[Section~6]{GS}, a weight $\lambda$ is not uniquely determined by its weight diagram and there are two conditions for $\tf_\l$ to be a weight diagram of a dominant weight $\l$. However, if two weights have the same diagram, they will be in the same block because they have the same atypicality and core (cf. \cite[Section~5]{GS}). Using the characterization of the blocks in $\F$ given in \cite[Section~5]{GS}, we have:
\begin{cor}\label{c1:blocks}
Let $\chi_{\lambda}$ be the central character corresponding to the weight $\lambda$ and $\F^{\chi_{\lambda}}$ be the corresponding block. Then $\lambda$ and $\lambda^{(l)}$, $l\in \Z$, belong to the same block $\F^{\chi_{\lambda}}$.
\end{cor}
\begin{proof}
It is sufficient to prove that $\lambda$ and $\lambda^{L}$ are in the same block. By the definition of the $L$-operator, $\lambda^L$ and $\lambda$ have the same atypicality.  Consider the three cases that defined $f_{\l}^L$ in \cite[Section~1.2]{CWZ} and suppose $(f_\l)_{-1}=\pm (f_\l)_i$ for some $1\leq i\leq n$. In the diagram of $\core(\lambda)$, we will have $<$ at the positions $-(f_\l)_j$, $j\neq i$. However, by the definition of $f_{\l}^L$, we will have $<$ at the same positions  $-(f_\l)_j$, $j\neq i$ in the diagram of $\core(\lambda^L)$. The corollary follows.
\end{proof}

In the following we have an explicit description of the block $\F^{\chi_{0}}$:
\begin{lem}\label{L1:principalblock}
$\F^{\chi_{0}}=\{0^{(l)}\,|\, l\in \Z\}.$
\end{lem}
\begin{proof}
 By Corollary~\ref{c1:blocks}, $0^{(l)}\in \F^{\chi_{0}}$ for all $l\in \Z$. Now assume $\lambda\in \F^{\chi_{0}}$, then $\atyp(\lambda)=1$ and $\core(\lambda)=\core(0)$. Thus the weight diagram of $\lambda$ has one $\times$, and $n-1$ $<$'s at positions $1,2,\ldots,n-1$ as in the weight diagram of $0$. The $\times$ cannot be at the zero position because $\lambda$ will not be dominant in that case, so it will be at some position $k\geq n$. The weight diagram of $\lambda$ is:
\[0,<,<,\ldots,<,0,\ldots,0,\times,\ldots,\]
where the $\times$ is at position $k$, from which we get:
\[\lambda+\rho=(-k|k,n-1,\ldots,2,1) \quad \text{or}\quad \lambda+\rho=(k|k,n-1,\ldots,2,1).\]
Thus we have two weights:
\[\lambda=(-k+n|k-n,0,\ldots,0) \quad \text{or}\quad \lambda=(k+n|k-n,0,\ldots,0).\]
Let $d=k-n\geq 0$. If $d=0$, then $\lambda=0$ or $\lambda=(2n|0,0,\ldots,0)$. By Lemma~\ref{t1:projcover(2,2n)}(c), $(2n|0,\ldots,0)^L=0$. Thus $\lambda=0^{(0)}$ or $\lambda=0^{(-1)}$.\\
If $d\geq 1$, then $\lambda=(-d|d,0,\ldots,0)$ or $\lambda=(2n+d|d,0,\ldots,0)$. By Lemma~\ref{t1:projcover(2,2n)}(a),(b), $(-d|d,0,\ldots,0)=0^{(d)}$ and $(2n+d|d,0,\ldots,0)^{(d+1)}=0$. Thus $\lambda=0^{(d)}$ or $\lambda=0^{(-d-1)}$. This completes the proof.
\end{proof}

\begin{lem}\label{L2:atypcomplexityKac}
Let $\lambda \in \F^{\chi_{\lambda}}$ with $\atyp(\lambda)=1$. Then $c_{\F}(K(\lambda))=c_{\F}(K(\mu))$ 
for some $\mu\in \F^{\chi_{0}}$.
\end{lem}

\begin{proof}
The translation functors defined in \cite[Section~5]{GS} move a simple module $L(\lambda)\in \F^{\chi}$ to $L(\mu)\in\F^{\tau}$. Let $\chi_{\lambda}$ be the central character corresponding to the weight $\lambda$ and $\F^{\chi_{\lambda}}$ be the corresponding block. Let $\F^{\chi_{0}}$ be the block containing the trivial module. The translation functors define an equivalence of blocks between $\F^{\chi_{\lambda}}$ and $\F^{\chi_{0}}$. To see this, we note that $\core(\lambda)$ has $<$ at $n-1$ positions, the same as $\core(0)$. Assume the $\times$ in $\tf_\lambda$ is at position $i$ for some $i\in \N$,
then the translation functors allow us to move the $<$'s to the positions $1,2,\ldots,n-1$ and move $\times$ to some position $k\geq n$. The diagram we get is $\tf_\mu$ where $\mu\in \F^{\chi_{0}}$. The same discussion as in  \cite[Section~6.3]{BKN1} shows that the translation functors preserve the complexity of any $\fg$-module. This completes the proof.
\end{proof}

\begin{thm}\label{t1:atypcomplexityKac}
For atypical $\lambda\in \x$, $c_{\F}(K(\lambda))=2n$.
\end{thm}

\begin{proof}
We use Lemmas~\ref{L2:atypcomplexityKac} and \ref{L1:principalblock} to get $c_{\F}(K(\lambda))=c_{\F}(K(\mu))$ for some $\mu=0^{(l)}\in \F^{\chi_{0}}, \, l\in \Z$. The proof follows by Lemma~\ref{L1:atypcomplexityKac} and Proposition~\ref{comptrivialKac} since $c_{\F}(K(0^{(l)}))=c_{\F}(K(0))=2n$.
\end{proof}

\section{geometric interpretation of complexity over $\osp(2|2n)$}\label{S:varKac}
 This section will be dedicated towards giving  a geometric interpretation of the complexity using rank, support, and associated varieties. In the following we introduce rank varieties and compute their dimensions for the Kac modules. In this section,  $\fg$ is assumed to be $\osp(2|2n)$ unless we indicate otherwise.  
\subsection{Rank variety}\label{ss:rankvar}
Let $\F(\fg_{\pm1})$ be the category of finite-dimensional $\fg_{\pm 1}$-modules and let $M\in \F(\fg_{\pm1})$. From \cite[Subsection~3.2]{BKN2}, we will use the following rank variety:
\[\fv_{\fg_{\pm1}}^{rank}(M):=\{x\in \fg_{\pm1} \mid \text{$M$ is not projective as a $U(\langle x \rangle)$-module}\}\cup\{0\},\]
where $U(\langle x \rangle)$ denotes the enveloping algebra of the Lie superalgebra generated by $x\in \fg_{\pm1}$. By \cite[Proposition~5.4]{BKN3}, $M$ will be projective as $U(\langle x \rangle)$-module if and only if it does not contain a direct summand which is isomorphic to the trivial module.

\subsection{Rank variety of Kac modules}
We start by computing the rank variety of an atypical Kac module over $\fg=\osp(2|2)$:

\begin{prop}\label{t2:varKac(2,2)}
For $\fg=\osp(2|2)$ and for an atypical $\lambda\in \x$, $\fv_{\fg_1}^{rank}(K(\lambda))=\fg_1$.
\end{prop}
\begin{proof}
By \cite[Subsection~3.8.4]{BKN2}, the action of $G_{\0}=\C^*\times \operatorname{Sp}(2n)$ on $\fg_1$ has two orbits $\{0\}$ and $\fg_1\setminus\{0\}$. Since $\fv_{\fg_1}^{rank}(M)$ is a closed $G_{\0}$-stable subvariety of $\fv_{\fg_1}^{rank}(\C)=\fg_1$, it is enough to find $0\neq x \in \fg_1$ such that $K(\lambda)$ contains a direct summand isomorphic to the trivial module when viewed as a $U(\langle x \rangle)$-module. Note that $\fg_1$ is spanned by the root vectors $x_1=E_{1,3}-E_{4,2}$ and $x_2=E_{1,4}+E_{3,2}$ with weights $wt(x_1)=\varepsilon_1-\delta_1$ and $wt(x_2)=\varepsilon_1+\delta_1$. On the other hand, $\fg_{-1}$ is spanned by $y_1=E_{2,4}+E_{3,1}$ and $y_2=E_{2,3}-E_{4,1}$ with weights $wt(y_i)=-wt(x_i),\, i=1,2$. We will need the elements $z_{ij}=[x_i,y_j]$, $i,j\in\{1,2\}$, of $\fg_{\0}\cong \C\oplus \s_2$ as we proceed:
\[
z_{11}=1\oplus h,\quad z_{22}=-1\oplus h,\quad z_{12}=0\oplus -2f, \quad\text{and}\quad z_{21}=0\oplus 2e,
\]
where $\{e,f,h\}$ is the standard basis of $\s_2$. The simple $\fg_{\0}$-modules are of the form $L_{\0}(\lambda)=\C_{\lambda_{-1}}\boxtimes L(d)$ where $L(d)$ is the simple $\s_2$-module of dimension $d+1$. Let $\{v_k\mid 0 \leq k \leq d\}$ be the basis for $L(d)$ given in \cite[Subsection~7.2]{HUM} with $v_0$ being the highest weight vector such that:
\[h.v_k=(d-2k)v_k,\quad e.v_k=(d-k+1)v_{k-1},\quad f.v_k=(k+1)v_{k+1},\]
with $v_{-1}=v_{d+1}=0$. Let $w_k=1\otimes v_k$ then $\{w_k\mid 0\leq k\leq d\}$ is a basis of $L_{\0}(\lambda)$ with $w_0$ being the highest weight vector. By the PBW-basis theorem for $U(\fg)$,  $\dim K(\lambda)=4(d+1)$ with basis $\{y_1^ay_2^b\otimes w_k\mid a,\,b\in \{0,1\},\,0\leq k\leq d\}$.  It is worth noting that the Kac module $K(\l)$ has a $\Z$-grading induced from the $\Z$-grading on $\fg$. Namely, the $\Z$-grading of the basis vector $y_1^ay_2^b\otimes w_k$ is $-a-b$. 

We have $\rho=(-1|1)$ then $\lambda=(\lambda_{-1}|d)$ is atypical if $(\rho+\lambda,\varepsilon_1\pm\delta_1)=0$. We have two cases:\\
\textbf{Case 1}: Using the odd root $wt(x_1)=\varepsilon_1-\delta_1$, we have $\lambda=(-d|d)$. Then, $x_1(1\otimes w_0)=1\otimes x_1w_0=0$ since $w_0$ is annihilated by $\fg_1$. Then by its $\Z$-grading, only $y_i\otimes w_k$, $i=1,2$ and $0\leq k\leq d$, could be sent by $x_1$ to the $U(\langle x_1 \rangle)$-submodule generated by $1\otimes w_0$.  If $y_i\otimes w_k$,  $i=1,2$ and $1\leq k\leq d$, is sent to that submodule, then
\[wt(y_i\otimes w_k)=wt(1\otimes w_0) -wt(x_1),\]
then $wt(y_i)+wt(x_1)=wt(w_0)-wt(w_k)$ is a sum of positive roots. Thus $i\neq 2$ since $wt(y_2)+wt(x_1)=-2\delta_1$. If $i=1$, then this sum of positive roots is $wt(y_1)+wt(x_1)=0$ which makes $k=0$. Therefore only $y_1\otimes w_0$ could land in this submodule. However,
\[x_1(y_1\otimes w_0)=-y_1\otimes x_1w_0+1\otimes z_{11}w_0=0+(-d+d)(1\otimes w_0)=0.\]
Therefore, the $U(\langle x_1 \rangle)$-submodule generated by $1\otimes w_0$ forms a trivial direct summand of $K(\lambda)$ when viewed as a $U(\langle x_1 \rangle)$-module.\\ 

\textbf{Case 2}: Using the odd root $wt(x_2)=\varepsilon_1+\delta_1$, $\lambda=(d+2|d)$. Then we can show
\begin{align*}
x_2(y_1y_2\otimes w_0)& =(-y_1x_2+z_{21})y_2\otimes w_0=-y_1(-y_2x_2+z_{22})\otimes w_0+(-2y_1+y_2z_{21})\otimes w_0\\
                      & \hspace{4.1cm} =  0-y_1\otimes (-d-2+d)w_0-2y_1\otimes w_0+0=0.
\end{align*}
In the above we used the fact that $x_2$ and $z_{21}$ both annihilate $w_0$. Since $y_1y_2\otimes w_0$ has a $\Z$-grading of $-2$ and $x_2$ has a $\Z$-grading of $1$, none of the basis elements $y_1^ay_2^b\otimes w_k$, $1\leq k\leq d, a,b=0,1$, can be sent by $x_2$ into the $U(\langle x_2 \rangle)$-submodule generated by $y_1y_2\otimes w_0$. Thus this submodule forms a trivial direct summand of $K(\lambda)$ when viewed as a $U(\langle x_2 \rangle)$-module. 
\end{proof}

To generalize the above theorem, we need to use the equivalence of blocks stated in \cite[Theorem~2]{GS}. Indeed, we have:
\begin{thm}\label{t:varkac}
For $\fg=\osp(2|2n)$ and when $\lambda$ is atypical, $\fv_{\fg_1}^{rank}(K(\lambda))=\fg_1$.
\end{thm}

\begin{proof}
We follow the same argument from \cite[Subsection~6.3]{BKN1}. Let $\chi_{\lambda}$ be the central character corresponding to the weight $\lambda$ and let $\F^{\chi_{\lambda}}$ be the corresponding block. \cite[Theorem~2]{GS} implies that $\F^{\chi_{\lambda}}$ is equivalent to the principal block of $\fg'=\osp(2|2)$ containing the trivial module. As discussed in \cite{GS}, the equivalence of these blocks is a composition of translation functors between the blocks of $\F$, followed by a restriction functor $\operatorname{Res}_{\mu}$ from $\fg$ to $\fg'$. Let $M$ be the image of $K(\l)$ under the translation functors. Following the proof of \cite[Theorem~6.3.1]{BKN1}, we can show that upon restricting to $\fg'$ we have, as $\fg'$-modules,
\[M=\operatorname{Res}_{\mu}(M)\oplus G_{\mu}(M),\]
for a $\fg'$-module $G_{\mu}(M)$ (cf. \cite[(4.21)]{BKN4}).  The composition of these functors takes Kac $\fg$-modules to Kac $\fg'$-modules and hence $\operatorname{Res}_{\mu}(M)$ is an atypical Kac $\fg'$-module. By Proposition~\ref{t2:varKac(2,2)} there exists $0\neq x' \in \fv_{\fg'_1}^{rank}(\operatorname{Res}_{\mu}(M))=\fg'_1$. By \cite[Theorem~6.6]{BKN3} we have
\[\fv_{\fg'_1}^{rank}(M)=\fv_{\fg'_1}^{rank}(\operatorname{Res}_{\mu}(M))\cup\fv_{\fg'_1}^{rank}(G_{\mu}(M)).\]
Thus there exists $0\neq x' \in \fg'\subseteq \fv_{\fg'_1}^{rank}(M)$. However, we can embed $\fg'\hookrightarrow\fg$ such that the odd elements $x_1,x_2,y_1,y_2\in \fg'$ given in the previous proof are sent respectively to the odd elements $E_{1,3}-E_{n+3,2},E_{1,n+3}+E_{3,2}, E_{2,3}-E_{n+3,1}, E_{2,n+3}+E_{3,1}\in \fg$. Therefore, we can find $0\neq x \in \fv_{\fg_1}^{rank}(M)$ but the $G_{\0}$-orbits in $\fg_1$ are $\{0\}$ and $\fg_1\setminus\{0\}$, hence $\fv_{\fg_1}^{rank}(M)=\fg_1$. 

Since the mentioned translation functors are equivalences of categories and by the properties of rank variety \cite[Proposition~6.3,\,Theorem~6.6]{BKN3}, we can show that \cite[Proposition~4.3]{BKN4} holds for rank varieties in the case of $\osp(2|2n)$. Thus rank varieties are preserved under these translation functors. Therefore,  $\fv_{\fg_1}^{rank}(K(\lambda))=\fv_{\fg_1}^{rank}(M)=\fg_1$.
\end{proof}

We can now interpret the complexity of Kac modules geometrically as follows:
\begin{cor}
For $\lambda\in \x$, $c_{\F}(K(\lambda))=\dim \fv_{\fg_1}^{rank}(K(\lambda)).$
\end{cor}
\begin{proof}
If $\l$ is typical, the Kac module $K(\l)$ is projective, then $\dim \fv_{\fg_1}^{rank}(K(\lambda))=0$ by \cite[Theorem~3.5.1]{BKN2} and the result follows by Proposition~\ref{typcomplexity}. If $\l$ is atypical, the result follows from Theorems~\ref{t1:atypcomplexityKac} and \ref{t:varkac}.
\end{proof}

After establishing the equivalence of blocks between $\fg$ and $\fg'$ as given in Theorem~\ref{t:varkac}, we can show that the operator $L$ (cf. Subsection~\ref{SS:simplecomplexity}) is onto:
\begin{prop}\label{Lsurjective}
The operator $L$ is surjective on the set of atypical weights in $\x$.
\end{prop}
\begin{proof}
First, we will show that $L$ is onto in the case of $\fg'=\osp(2|2)$. Let $\mu$ be an atypical weight, then $\mu=(d+2|d)$ or $\mu=(-d|d)$ for some $d\geq 0$. By Lemma~\ref{t1:projcover(2,2n)}, $(0|0)=(2|0)^L$, $(d+2|d)=(d+3|d+1)^L$ for any $d\geq 0$ and $(-d|d)=(-d+1|d-1)^L$ for any $d\geq 1$. Thus $\mu$ is the image under $L$ of some atypical weight.

Now let $\mu$ be an atypical weight in the case of $\fg=\osp(2|2n)$. Let $L(\mu')$ be the image of the simple module $L(\mu)$ under the above equivalence of blocks between $\fg$ and $\fg'$. Then $\mu'$ is an atypical weight for $\fg'$. Thus there exists an atypical weight $\lambda'$ with $\lambda'^L=\mu'$. The head of $K(\lambda')$ is $L(\lambda')$ and the socle is $L(\lambda'^L)=L(\mu')$. This Kac module corresponds to a Kac module $K(\lambda)$ (where $\lambda$ corresponds to $\lambda'$ under the same equivalence) which has $L(\lambda)$ as its head and $L(\mu)$ as its socle. But the socle of $K(\lambda)$ is $L(\lambda^L)$, thus $L(\lambda^L)\cong L(\mu)$ which shows $\mu=\lambda^L$ for some atypical $\lambda$.
\end{proof}

We obtain a uniform geometric interpretation of the complexity of the Kac and simple modules by combining support and associated varieties. These varieties were defined in Subsections~\ref{suppvar} and \ref{assocvar}. First, we find the dimensions of the associated varieties of the simple and Kac modules:

\begin{prop}\label{assocvarsimple}
For $\lambda\in \x$, let $X(\lambda)$ be $L(\lambda)$ or $K(\lambda)$, then
 $\dim \X_{X(\lambda)}=2n,0$ if $\atyp(\lambda)=1,0$ respectively.
\end{prop}

\begin{proof}
When $\lambda$ is typical, $L(\lambda)$ and $K(\lambda)$ are projective, hence $\dim \X_{X(\lambda)}=0$ by \cite[Theorem~3.4]{DS}. For the rest of the proof, assume $\lambda$ is atypical. \cite[Corollary~2.5]{S} implies that $ \X_{L(\lambda)}=\X$ when $\lambda$ is atypical. In $\osp(2|2n)$, $\X$ has two irreducible components each with dimension equal to $\displaystyle{\frac{\dim \fg_{\1}}{2}}=2n$. Thus $\dim \X=2n$ \cite[Corollary~4.8,~4.9]{DS}.

Now assume $X(\lambda)=K(\lambda)$. From the definitions of the rank variety and the associated variety, we have
$\fv_{\fg_1}^{rank}(K(\lambda))=\X_{K(\lambda)}\cap \fg_1\subseteq \X_{K(\lambda)}$. Following the same proof of \cite[Theorem~6.4.1]{BKN1}, we show that the inclusion is in fact an equality. Thus $\dim \X_{K(\lambda)}=\dim \fv_{\fg_1}^{rank}(K(\lambda))=\dim \fg_1=2n$,   (using Theorem~\ref{t:varkac}).
\end{proof}

Second, we give the dimensions of the support varieties of the simple and the Kac modules:
\begin{prop}\label{suppvarsimpleKac}
	For $\lambda\in \x$,
	\begin{itemize}
	\item  $\dim \fv_{(\fg,\fg_{\0})}(L(\lambda))=\atyp(\lambda)$, \cite[Corollary~4.4.2]{kujawa}.
	\item   $\dim \fv_{(\fg,\fg_{\0})}(K(\lambda))=0$, \cite[Corollary~3.2]{BKN4}.
	\end{itemize}
\end{prop}

Combining the results from Propositions~\ref{typcomplexity}, \ref{assocvarsimple}, \ref{suppvarsimpleKac} and Theorems~\ref{t2:atypcomplexitysimple}, \ref{t1:atypcomplexityKac} we conclude that:
\begin{thm}\label{ospgeometric1}
For $\lambda\in \x$, let $X(\lambda)$ be $L(\lambda)$ or $K(\lambda)$, then $$c_{\F}(X(\lambda))=\dim \X_{X(\lambda)}+\dim \fv_{(\fg,\fg_{\0})}(X(\lambda)).$$
\end{thm}

\section{$z$-complexity over $\osp(2|2n)$}\label{zcomplexity}
Let $\fg=\osp(2|2n)$. The $z$-complexity is defined in Subsection~\ref{zcompdef}. 
\subsection{$z$-complexity of simple modules}
We start by computing the $z$-complexity of the trivial $\osp(2|2n)$-module.
\begin{lem}\label{t1:zcompsimple}
 $z_{\F}(\C)=2$.
\end{lem}

\begin{proof}
The proof follows directly by counting the number of summands in each $P_d$ in the resolution \eqref{projrestrivial} of the trivial module. By doing so, we can see that the number of summands in each $P_d$ is a polynomial in $d$ of degree $1$. Hence, the rate of growth of this number is $s(\p)=2$. The result follows.
\end{proof}

To find the $z$-complexity of all atypical simple modules, we will show that simple modules of the same atypicality have the same $z$-complexity.

\begin{lem}\label{l:bdkac}
For any module $X\in \F$ and an atypical Kac module $K(\lambda)$, there exists a constant $D_X$ depending only on $X$ such that
\[K(\lambda)\otimes X\cong \bigoplus_{\gamma\in I}K(\gamma),\]
where $|I|\leq D_X$.
\end{lem}

\begin{proof}
By the definition of Kac modules, we have
\[K(\lambda)=U(g)\otimes_{U(\fp^+)}L_{\0}(\lambda).\]
As a $\fg_{\0}$-module, $L_{\0}(\lambda)\otimes X$ decomposes into a direct sum $\bigoplus_{\mu\in I_X} L_{\0}(\lambda)\otimes L_{\0}(\mu)$, where $I_X$ is a finite indexing set depending only on $X$. By the generalized Littlewood-Richardson formula \cite[Subsection~A.4]{L}, each summand $L_{\0}(\lambda)\otimes L_{\0}(\mu)$ decomposes into a finite direct sum $\bigoplus_{\gamma\in I_{\lambda,\mu}} L_{\0}(\gamma)$ indexed by the $\mathfrak{sp}(2n)$-standard Young tableaux of shape $p(\mu)$ which are $\lambda$-dominant (see \cite[Subsection~A.4]{L} for definitions). Let $l(\mu)$ be the number of boxes in $p(\mu)$, then
 \[|I_{\lambda,\mu}|\leq (2n)^{l(\mu)}.\]
  By using the functor $U(g)\otimes_{U(\fp^+)}-$, we have the following isomorphisms as $U(\fg)$-modules:
\[K(\lambda)\otimes X\cong \bigoplus_{\mu\in I_X}\bigoplus_{\gamma\in I_{\lambda,\mu}} U(g)\otimes_{U(\fp^+)}L_{\0}(\gamma)\cong \bigoplus_{\mu\in I_X}\bigoplus_{\gamma\in I_{\lambda,\mu}} K(\gamma).\]
The number of summands is at most $D_X:=\sum_{\mu\in I_X}(2n)^{l(\mu)}$. This completes the proof.
\end{proof}

\begin{cor}
Let $\lambda\in \x$ be an atypical weight and $P(\lambda)$ be the projective cover of $L(\lambda)$.  For any $X\in \F$, there exists a constant $E_X$ depending only on $X$ such that
\[P(\lambda)\otimes X\cong \bigoplus_{\alpha\in J}P(\alpha),\]
where $|J|\leq E_X$.
\end{cor}
\begin{proof}
Let $\lambda'$ be the atypical weight as in Proposition~\ref{Lsurjective} satisfying $\lambda'^L=\lambda$. By the Kac flag \eqref{KacFlag}, we have:
\[0\rightarrow K(\lambda')\rightarrow P(\lambda'^L)\rightarrow K(\lambda'^L)\rightarrow 0.\]
Thus
\[0\rightarrow K(\lambda')\rightarrow P(\lambda)\rightarrow K(\lambda)\rightarrow 0.\]
 This Kac filtration of $P(\l)$ together with the exactness of the tensor functor (over $\C$) imply the exact sequence:
\[0\rightarrow K(\lambda')\otimes X\rightarrow P(\lambda)\otimes X\rightarrow K(\lambda)\otimes X\rightarrow 0.\]
The factors $K(\lambda)\otimes X$ and $K(\lambda')\otimes X$ decompose into Kac modules by the previous lemma. Thus $P(\lambda)\otimes X$ has a Kac filtration. By Lemma~\ref{l:bdkac} there exists a constant $D_X$ such that the number of Kac modules in this filtration is at most $E_X:=2D_X$. On the other hand, the projective module $P(\lambda)\otimes X$ decomposes into a finite direct sum of projective indecomposables $\bigoplus_{\alpha\in J} P(\alpha)$ where each summand has a Kac filtration. However Kac filtrations of the same module will have the same number of Kac modules. In fact, by \cite[Proposition~3.3]{BKN4}, the number of times a Kac module $K(\mu)$ appears in a Kac filtration of a module $M$ is equal to $\dim \Hom(M, K'(\mu))$, where $K'(\mu)$ is the dual Kac module (cf. \cite[Subsection~3.4]{BKN4}). This dimension is not dependent on the choice of the filtration. Thus the number of the projective indecomposables $P(\alpha)$ is bounded by $E_X$.
\end{proof}

\begin{lem}\label{t:zcomp}
Let $M,\, N,\,X,\, T \in \F$ such that $M\otimes X\cong N\oplus T$. Then $z_{\F}(M)\geq z_{\F}(N)$.
\end{lem}

\begin{proof}
Let $\p\rightarrow M\otimes X$ and $Q_\bullet \rightarrow M$ be the minimal projective resolutions of $M\otimes X$ and $M$ respectively. Then $Q_\bullet\otimes X\rightarrow M\otimes X$ is a projective resolution of $M\otimes X$. Using the above corollary, the $d$th term in this resolution decomposes as follows:
\[Q_d\otimes X \cong \Bigg(\bigoplus_{\lambda\in R_d} P(\lambda)\Bigg)\otimes X\cong \bigoplus_{\lambda\in R_d}\bigoplus_{\alpha\in J}P(\a),\]
where $|J|\leq E_X$. Thus, for each $d$, the number of summands in $Q_d\otimes X$ is at most $|R_d|\cdot E_X$ which is a constant multiple of the number of summands in $Q_d$. Recall that $s(\p)$ denotes the rate of growth of the number of summands in $P_d$, the $d$th term in the resolution $\p$, then by minimality of the resolution, we have $s(\p)\leq s(Q_\bullet\otimes X)$. Thus,
\[z_{\F}(M\otimes X)=s(\p) \leq s(Q_\bullet\otimes X) = s(Q_\bullet)= z_{\F}(M).\]

By using the definition of the $z$-complexity we can easily show that
\[z_{\F}(N\oplus T)=\operatorname{max}\Big(z_{\F}(N),\,z_{\F}(T)\Big).\]
Then
\[z_{\F}(M)\geq z_{\F}(M\otimes X)=z_{\F}(N\oplus T) \geq z_{\F}(N).\]
\end{proof}

\begin{thm}\label{t2:zcompsimple}
Let $\lambda \in \x$.
\begin{enumerate}
\item [(1)] If $\lambda$ is typical, then $z_{\F}(L(\lambda))=0$.\\
\item [(2)] If $\l$ is atypical, then $z_{\F}(L(\lambda))=2$.
\end{enumerate}
\end{thm}

\begin{proof}
If $\lambda$ is typical, then $L(\lambda)$ is projective, and hence $z_{\F}(L(\lambda))=0$. Let $\l$ be atypical, then \cite[Corollary~3.2.2]{kujawa} implies that if $\lambda$ and $\mu$ have the same atypicality, there are modules $X_1$ and $X_2$ such that $L(\mu)$ is a direct summand of $L(\lambda)\otimes X_1$ and $L(\lambda)$ is a direct summand of $L(\mu)\otimes X_2$. Then Lemma~\ref{t:zcomp} shows that $z_{\F}(L(\lambda))=z_{\F}(L(\mu))$. Thus, for an atypical $\lambda$, $z_{\F}(L(\lambda))=z_{\F}(\C)=2$.
\end{proof}

\subsection{$z$-complexity of Kac modules}
\begin{thm}\label{t:zcompkac}
Let $\lambda \in \x$.
\begin{enumerate}
\item [(1)] If $\lambda$ is typical, then $z_{\F}(K(\lambda))=0$.\\
\item [(2)] If $\lambda$ is atypical, then $z_{\F}(K(\lambda))=1$.
\end{enumerate}
\end{thm}

\begin{proof}
If $\lambda$ is typical, then $K(\lambda)$ is projective, and hence $z_{\F}(K(\lambda))=0$. Assume $\lambda$ is atypical. For $\l=0$, we can use the projective resolution \eqref{projresK(0)} to see that $z_{\F}(K(\lambda))=1.$ Following the same proof as in Lemma~\ref{L1:atypcomplexityKac}, we can show that $z_{\F}(K(\lambda))=z_{\F}(K(\lambda^{L}))$ thus $z_{\F}(K(0^{(l)}))=1$ for any $l\in \Z$ which means that  $z_{\F}(K(\lambda))=1$ if $\lambda\in\F^{\chi_{0}}$. Since the translation functors between the atypical blocks are categorical equivalences and the $z$-complexity is a categorical invariant, then $z_{\F}(K(\lambda))=1$ for any atypical $\lambda$.
\end{proof}

\subsection{Detecting subsuperalgebra}
As we interpreted the complexity of the simple and the Kac modules geometrically, we also find a geometric interpretation of the $z$-complexity. To do so, a detecting subalgebra is introduced. Let $\ff_{\1}\subseteq \fg_{\1}$ be the span of the root vectors $x_\alpha, x_{-\alpha}$ where $\alpha=\varepsilon_1-\delta_1$. In the matrix realization, $x_\alpha=E_{1,3}-E_{n+3,2}$ and $x_{-\alpha}=E_{2,n+3}+E_{3,1}$. Set $\ff_{\0}=[\ff_{\1},\ff_{\1}]$. Then $\ff_{\0}$ is spanned by the diagonal matrix $E_{1,1}-E_{2,2}+E_{3,3}-E_{n+3,n+3}$. We define a three-dimensional subalgebra of $\fg$ by $\ff:=\ff_{\0}\oplus \ff_{\1}$. The Lie superalgebra $\ff$ is classical and so has a support variety theory. Furthermore, as $[\ff_{\0},\ff_{\1}]=0$, it follows that these varieties admit a rank variety description and, in particular, can be identified as subvarieties of $\ff_{\1}$, i.e.,
\begin{equation*}
\fv_{(\ff,\ff_{\0})}(M)=\fv_{\ff_{\1}}^{rank}(M)=\{y\in \ff_{\1}\mid \text{$M$ is not projective as $U(\langle y \rangle)$-module}\}\cup\{0\}.
\end{equation*}
For example, $\fv_{(\ff,\ff_{\0})}(\C)=\ff_{\1}$. If $M$ is projective, then $\fv_{(\ff,\ff_{\0})}(M)=0$.

\begin{prop}\label{t:vardetect}
Let $\lambda \in \x$.
\begin{enumerate}
\item [(1)] If $\lambda$ is typical, then $\fv_{(\ff,\ff_{\0})}(K(\lambda))=\fv_{(\ff,\ff_{\0})}(L(\lambda))=0.$
\item [(2)] If $\lambda$ is atypical, then $\dim \fv_{(\ff,\ff_{\0})}(K(\lambda))=1,$ and $\dim \fv_{(\ff,\ff_{\0})}(L(\lambda))=2.$
\end{enumerate}
\end{prop}

\begin{proof}
 If $\lambda$ is typical, then $L(\lambda)$ and $K(\lambda)$ are projective and the result follows. Let $\lambda$ be an atypical weight. As argued in \cite[Theorems~6.4.1,~9.2.1]{BKN1} we can show $$\fv_{(\ff,\ff_{\0})}(K(\lambda))=\ff_{\1}\cap \fg_{1},$$
  which implies $\dim \fv_{(\ff,\ff_{\0})}(K(\lambda))=1.$ Since $\l$ and $(0|0,\ldots,0)$ have the same atypicality, \cite[Theorem~4.1.1]{kujawa} implies that
\[\fv_{(\ff,\ff_{\0})}(L(\lambda))=\fv_{(\ff,\ff_{\0})}(L(0|0,\ldots,0))=\ff_{\1},\]
 thus $\dim \fv_{(\ff,\ff_{\0})}(L(\lambda))=2$.
\end{proof}

By Theorems~\ref{t2:zcompsimple}, \ref{t:zcompkac}, and Proposition~\ref{t:vardetect}, the following geometric interpretation of the $z$-complexity holds:
 \begin{thm}\label{ospgeometric2}
 If $X(\lambda)$ is a simple or a Kac module over $\osp(2|2n)$, then
  \[z_{\F}(X(\lambda))=\dim \fv_{(\ff,\ff_{\0})}(X(\lambda)).\]
 \end{thm}

\section{Additional examples}\label{s:additionalexamples}
In this section we assume that $\fg$ is either $\osp(3|2)$, $\dd$ for some $\a\in\C\setminus\{0,-1\}$, $G(3)$, or $F(4)$. In these cases, $\fg_{\0}$ is semisimple and hence our category $\F$ is the category of finite-dimensional supermodules. In this section we refer to \cite{Germoni} and \cite{Mart} to describe the atypical blocks over these Lie superalgebras. We will also use the notation developed in \cite{Germoni} for $\osp(3|2)$, $\dd$, $G(3)$ and the notation in \cite{Mart} for $F(4)$. However, in this section, $\N$ indicates the set of positive integers.

Germoni \cite{Germoni} denotes by $\fg$-mod the category of finite-dimensional $\fg$-modules with even morphisms of representations. In \cite[Lemma~1.1.1]{Germoni}, the category $\gmod$ is shown to contain enough projective modules where projective and injective modules coincide. However, we will be considering the category $\F$ of finite-dimensional $\fg$-modules with all morphisms. The choice of even morphisms has no effect on the radical layer structure of the projective indecomposable modules. We first provide the necessary background on these Lie superalgebras and we then compute the complexity and the $z$-complexity of the atypical simple modules. The typical ones are projective, and hence they have zero complexity and $z$-complexity. We also interpret these complexities geometrically as before.

\subsection{Case I: $\osp(3|2)$}

Let $\fg=\osp(3|2)$ (see \cite{Kac}). The even part of $\fg$ is $\fg_{\0} \cong
\mathfrak{so}(3)\oplus\mathfrak{sp}(2)\cong \s_2\oplus\s_2$,
hence $\dim\fh^*=2$ with a basis $\{\eps,\d\}$. The odd part $\fg_{\1}$ is the $6$-dimensional $\fg_{\0}$ module $\C^3\boxtimes (\C^2)^*$ where $\C^3$ is the natural $\mathfrak{so}(3)$-module and $(\C^2)^*$ is the right $\mathfrak{sp}(2)$-module of row vectors.

The atypical dominant weights are denoted by $\l_0=(0,0)$ and $\l_l=(l-1,l)$ for $l\in \N$. The trivial module $L(\l_0)$ will be denoted by $\C$. According to (\cite[Theorem~2.1.1]{Germoni}), the principal block $\Ga_0=\{\l_l\tq l\geq 0\}$
is the unique atypical block.

Let $P(\mu)$ be the projective cover of $L(\mu)$. As a $\fg_{\0}$-module, $L(\mu)$ contains a simple $\fg_{\0}$-module $L_{\0}(\mu)$ as a composition factor. Using the discussion in \cite[Subsection~5.1]{BKN1}, we have the following bounds:
\begin{equation}\label{eq1:Bdprojcover}
\dim L_{\0}(\mu)\leq \dim P(\mu)\leq 2^{\dim \fg_{\1}}\dim L_{\0}(\mu),
\end{equation}
Note that $L_{\0}(\l_0)$ is the one-dimensional trivial $\fg_{\0}$-module. For $l\in \N$, $L_{\0}(\l_l)=V_{l-1}\boxtimes V_l$ where $V_m$ denotes the simple $\s_2$-module of dimension $m+1$. Thus, for $l\in \N$,
\begin{equation}\label{eq1':Bdprojcover}
l(l+1)\leq \dim P(\l_l)\leq 2^6 l(l+1).
\end{equation}

\subsection{Case II: $\dd$} \label{D}
Let $\alpha \in\C\setminus\{0,-1\}$, and let $\fg$ be the basic classical Lie superalgebra
$\dd$  (see \cite{Kac}). The even part of $\fg$ is $\fg_{\0}\cong \s_2\oplus\s_2\oplus \s_2$. The odd part of $\fg$ is $\fg_{\1}\cong \C^2\boxtimes \C^2\boxtimes \C^2$ where $\C^2$ is the $2$-dimensional natural $\s_2$-module. 

Atypical dominant weights are described in \cite[Subsection~3.1]{Germoni}. If $\a\not\in\Q$, the only atypical dominant weights are $\l_0=(0,0,0)$ (corresponding to the trivial module) and $\l_l=(l+1,l-1,l-1)$ for $l\in \N$. Now assume $\a\in\Q$. For $k\in \N$ and $l\in \Z$, let $\l_{k,l}$ be the atypical weight defined in \cite[Subsection~3.1]{Germoni}. Set $\Ga_k=\{\l_{k,l}\tq l\in\Z\}$. According to \cite[Theorem~3.1.1]{Germoni}, the principal block in both cases is $\Ga_0=\{\l_{l}\tq l\geq 0\}$ and it is equivalent to the principal block of $\osp(3|2)$. Moreover, the projective covers of the simple modules in the principal block of $\dd$ have the same radical layer structure as those over $\osp(3|2)$. In addition, if $\alpha\notin \Q$, then $\Ga_0$ is the unique atypical block. If $\alpha\in \Q$, any atypical block is one of $\Ga_k$, $k\geq 0$. 

 Let $L_{\0}(\mu)$ be the simple $\fg_{\0}$-module of highest weight $\mu$. Note that $L_{\0}(\l_0)$ is the trivial $\fg_{\0}$-module and for $\l_l=(l+1,l-l,l-1)$, $L_{\0}(\l_l)=V_{l+1}\boxtimes V_{l-1}\boxtimes V_{l-1}$ where $V_m$ denotes the simple $\s_2$-module of dimension $m+1$. Then by \eqref{eq1:Bdprojcover} we have for $l\in \N$:
\begin{equation}\label{eq1:DBdprojcover}
(l+2)(l^2)\leq \dim P(\l_l)\leq 2^8 (l+2)(l^2).
\end{equation}

\subsection{Case III: $G(3)$}\label{G}
Let $\fg$ be the basic classical Lie superalgebra $G(3)$ (see \cite{Kac}). The even part of $\fg$ is~:
$\fg_{\0}\cong \s_2\oplus G_2$ where $G_2\subsetneqq \mathfrak{so}_7$ is the exceptional Lie algebra. The odd part of $\fg$ is $\fg_{\1}\cong \C^2\boxtimes \C^7$ where $\C^2$ is the $2$-dimensional natural $\s_2$-module and $\C^7$ is the $7$-dimensional non-trivial $G_2$-module of minimal dimension. 

We identify the set of dominant weights for $\fg_{\0}$ with
nonnegative integer combinations of the fundamental
weights $(\d;\w_1=\e_1+\e_2;\w_2=\e_1+2\e_2)$. For $\nu=m_1\omega_1+m_2\omega_2=(m_1,m_2)$, the dimension of the simple $G_2$-module $L_{G_2}(\nu)$ can be computed by the Weyl-dimension formula. Note that the fundamental weights given in \cite[Section~24.3]{HUM} are $\l_1:=\omega_2,\l_2=:\omega_1$. The formula given in \cite[Section~24.3]{HUM} is being adjusted to this setting:
\begin{equation}\label{Gdimformula}
\dim L_{G_2}(\nu)=\frac{1}{5!}(m_2+1)(m_1+1)(m_2+m_1+2)(m_2+2m_1+3)(m_2+3m_1+4)(2m_2+3m_1+5).
\end{equation}

For $k$ and $l$ nonnegative integers, let $\l_{k,l}$ be the atypical weight defined in \cite[Subsection~4.1]{Germoni}. Set $\Ga_k=\{\l_{k,l}\tq l\geq 0\}$.  Note that $\l_{0,0}=(0,0,0)$, $\l_{0,l}=(l+4,l-1,0)$ for $l\in \N$, and $\Ga_0$ is the principal block. By \cite[Theorem~4.1.1]{Germoni}, every atypical block of $\fg$ is one of the $\Ga_k$ which is equivalent to the principal block of $\osp(3|2)$. Moreover, the projective covers of the simple modules in $\Ga_k$ have the same radical layer structure as those in the principal block of $\osp(3|2)$.

Let $L_{\0}(\mu)$ be the simple $\fg_{\0}$-module of highest weight $\mu$. Note that $L_{\0}(\l_{0,0})$ is the trivial $\fg_{\0}$-module and, for $l \in \N$, $L_{\0}(\l_{0,l})=V_{l+4}\boxtimes L_{G_2}(l-1,0)$ where $V_{l+4}$ denotes the simple $\s_2$-module of dimension $l+5$ and $L_{G_2}(l-1,0)$ is the simple $G_2$-module of dimension given by \eqref{Gdimformula}. Then by \eqref{eq1:Bdprojcover} we have for $l\in \N$:
\begin{equation}\label{eq1:GBdprojcover}
\frac{1}{5!}(l+5)(l)(l+1)(2l+1)(3l+1)(3l+2)\leq \dim P(\l_l)\leq  \frac{2^{14}}{5!}(l+5)(l)(l+1)(2l+1)(3l+1)(3l+2).
\end{equation}

\subsection{Case IV: $F(4)$}\label{F}
Let $\fg$ be the basic classical Lie superalgebra $F(4)$ (see \cite{Kac}). The even part of $\fg$ is $\fg_{\0}\cong \mathfrak{so}_7\oplus \s_2$. The odd part of $\fg$ is $\fg_{\1}\cong \mathfrak{spin}_7\boxtimes \C^2$ where $\C^2$ is the $2$-dimensional natural $\s_2$-module and $\mathfrak{spin}_7$ is the $8$-dimensional spin representation of $\mathfrak{so}_7$.

The set of dominant weights and the atypical blocks are described in \cite{Mart}. The fundamental weights are $\omega_1=\e_1,\,\omega_2=\e_1+\e_2,\,\omega_3=(1/2)(\e_1+\e_2+\e_3),$ and $\omega_4=(1/2)\delta$. For $\nu=m_1\omega_1+m_2\omega_2+m_3\omega_3$, the dimension of the simple $\mathfrak{so}_7$-module, $L_{\mathfrak{so}_7}(\nu)$, can be computed by the Weyl-dimension formula \cite[Section~24.3]{HUM}:
\begin{multline}\label{Fdimformula}
\dim L_{\mathfrak{so}_7}(\nu)=\frac{1}{720}(m_1+1)(m_2+1)(m_3+1)(m_1+m_2+2)(m_2+m_3+2)\\
(2m_2+m_3+3)(m_1+m_2+m_3+3)(m_1+2m_2+m_3+4)(2m_1+2m_2+m_3+5).
\end{multline}

As described in \cite{Mart}, the block containing the trivial module is denoted by $F^{(1,1)}$. From \cite[Lemma~7.1]{Mart}, 
\[F^{(1,1)}=\{\l _1=(0,0,0|0),\l _2=(0,0,0|3), \l_l=(1/2)(l+1,l+1,l+1|l+7)\,\forall\, l\neq 1,2\},\]
where the components are the coefficients of $\e_1,\e_2,\e_3,\d$ respectively. Thus for $l\neq 1,2$, $\l_l= (l+1)\omega_3+(l+7)\omega_4$. 

Let $L_{\0}(\mu)$ be the simple $\fg_{\0}$-module of highest weight $\mu$. Note that $L_{\0}(\l_1)$ is the trivial $\fg_{\0}$-module and for $l \neq 1,2$, $L_{\0}(\l_{l})=L_{\mathfrak{so}_7}((l+1)\omega_3)\boxtimes V_{l+7}$ where $V_{l+7}$ denotes the simple $\s_2$-module of dimension $l+8$ and $L_{\mathfrak{so}_7}((l+1)\omega_3)$ is the simple $\mathfrak{so}_7$-module of dimension given by \eqref{Fdimformula}. Then by \eqref{eq1:Bdprojcover} we have for $l\neq 1,2$:
\begin{equation}\label{eq1:FBdprojcover}
\frac{1}{360}(l+2)(l+3)(l+4)^2(l+5)(l+6)(l+8)\leq \dim P(\l_l)\leq  \frac{2^{16}}{360}(l+2)(l+3)(l+4)^2(l+5)(l+6)(l+8).
\end{equation}

\subsection{Complexity of simple modules}
We give a minimal projective resolution of the trivial module $\C$ to compute its complexity, then we use the generalized Kac-Wakimoto conjecture to show that any atypical simple module will have the same complexity as the trivial module.
\begin{thm}\label{osp(3|2)DGFsimple}
Let $S$ be an atypical simple $\fg$-module. We have $c_{\F}(S)=4,5,8,9$ if $\fg=\osp(3|2),\dd,G(3),F(4)$ respectively.
\end{thm}

\begin{proof}
Let $\fg=\osp(3|2)$. Using the radical layer structure of the projective indecomposable modules given in \cite[Theorem~2.1.1]{Germoni} and following the same diagrammatic techniques used in proving Theorem~\ref{t2:atypcomplexitysimple}, the minimal projective resolution of the trivial module is given by:
\begin{align}\label{projres1}
 \cdots \rightarrow P_d\rightarrow \cdots \rightarrow P_0=P(\l_0)\rightarrow \C \rightarrow 0,
\end{align}
where the $d$th term,  $d\geq 1$, in this resolution is given by:
\[P_d=\begin{cases}
 P(\l_{d+1})\oplus P(\l_{d-1})\oplus \cdots \oplus P(\l_2)\quad \text{if $d$ is odd,}\\
 P(\l_{d+1})\oplus P(\l_{d-1})\oplus \cdots \oplus P(\l_3)\oplus P(\l_0)\quad \text{if $d\equiv 0 \mod 4$,}\\
 P(\l_{d+1})\oplus P(\l_{d-1})\oplus \cdots \oplus P(\l_3)\oplus P(\l_1)\quad \text{if $d\equiv 2 \mod 4$.}
    \end{cases}\]
Then, by \eqref{eq1':Bdprojcover}, we have for all $d\geq 1$:
\begin{equation*}
C'\cdot d^3\leq \dim P_d \leq C\cdot d^3,
\end{equation*}
for some positive constants $C,\, C'$.  This shows that $c_{\F}(\C)=4$ over $\osp(3|2)$. 

Let $\fg=\dd$ or $G(3)$. Since the principal block $\Ga_0$ in both cases is equivalent to the principal block of $\osp(3|2)$ and the projective covers have the same structures in both cases, we will have the same minimal projective resolution \eqref{projres1} for the trivial module $\C$. Using the bounds in equations~\eqref{eq1:DBdprojcover} and \eqref{eq1:GBdprojcover}, we have for all $d\geq 1$:
\begin{align*}
	D'\cdot d^4\leq \dim P_d\leq D\cdot d^4,\quad \text{if $\fg=\dd$}\\
	G'\cdot d^7\leq \dim P_d\leq G\cdot d^7, \quad \text{if $\fg=G(3)$}
\end{align*}
for some positive constants $D,\, D',\,G,\, G'$. This shows that $c_{\F}(\C)=5$ over $\dd$ and  $c_{\F}(\C)=8$ over $G(3)$.

Let $\fg=F(4)$. The description of the projective indecomposable modules over $F(4)$ given in \cite[Lemma~11.1]{Mart} is the same as the one given in \cite{Germoni} over $G(3)$, except for the difference in notation.  This similarity means that the projective resolution \eqref{projres1} of the trivial module over $\osp(3|2)$ will carry over to $F(4)$. Using the bounds in ~\eqref{eq1:FBdprojcover}  we have for $d\geq 1$:
	\begin{align*}
	F'\cdot d^8\leq \dim P_d\leq F\cdot d^8,
	\end{align*}
	for some positive constants $F,\, F'$. Thus $c_{\F}(\C)=9$ over $F(4)$.

By \cite[Theorem~4.1.1]{kujawa}, all simple $\osp(3|2)$-modules of the same atypicality have the same complexity. Thus the complexity of all atypical simple $\osp(3|2)$-modules is $4$.
On the other hand, by \cite[Example~3.4]{KW}, we see that the Kac-Wakimoto conjecture holds for simple modules over $\dd,\,G(3),\,F(4)$. Since this conjecture holds, then \cite[Theorem~4.1.1]{kujawa} holds over these Lie superalgebras. Thus all atypical simple modules over these Lie superalgebras have the same complexity as the trivial module. The result follows.
\end{proof}

\subsection{Geometric interpretation of the complexity}
Using support and associated varieties, we interpret the above complexities geometrically:

\begin{thm}\label{GMosp(3|2)}
		Let $\fg=\osp(3|2),\dd,\,G(3),$ or $F(4)$. If $S$ is a simple $\fg$-module, then
	$$c_{\F}(S)=\dim \X_{S}+\dim \fv_{(\fg,\fg_{\0})}(S).$$
\end{thm}
\begin{proof}
	If $S$ is typical then S is projective and both sides are zero. Let $S$ be an atypical simple module. By \cite[Corollary~2.5]{S}, we have $\X_{S}=\X$ by the stratification described in \cite[Section~2]{S}. Using \cite[Theorem~4.5, Corollary~4.8]{DS},
	$\dim \X=3,4,7,8$ if $\fg=\osp(3|2),\dd,G(3),F(4)$ respectively.
	On the other hand, by \cite[Corollary~4.4.2]{kujawa} we have $\dim \fv_{(\fg,\fg_{\0})}(S)=\atyp(S)=1$. The result then follows from Theorem~\ref{osp(3|2)DGFsimple}.
\end{proof}

\subsection{$z$-complexity of simple modules}
 We start by computing the $z$-complexity of the trivial module over $\osp(3|2),\,\dd,\,G(3)$, and $F(4)$:
 \begin{prop}
 	$z_{\F}(\C)=2$.
 \end{prop}
 \begin{proof}
 	Using the resolution in the proof of Theorem~\ref{osp(3|2)DGFsimple}, the number of summands in $P_d$ is either $d/2+1$ if $d$ is even or $(d+1)/2$ if $d$ is odd. The rate of growth of this number is $2$. The result follows.
 \end{proof}
We then compute the $z$-complexity of the other simple $\osp(3|2)$-modules in the unique atypical block $\Gamma_0$: 
 \begin{thm}
 	Let $\fg=\osp(3|2)$. For $l\geq 1$, we have $z_{\F}(L(\l_l))=2$.
 \end{thm}
 
 \begin{proof}
 Since the projective cover of $L(\l_1)$ has the same structure as the projective cover of $L(\l_0)=\C$, one can write an analogous minimal projective resolution to show that $z_{\F}(L(\l_1))=2$. The minimal projective resolution of $L(\l_2)$ is given by:
 \begin{align}
 \cdots \rightarrow P_d\rightarrow \cdots \rightarrow P_0=P(\l_2)\rightarrow L(\l_2) \rightarrow 0,
 \end{align}
 where the $d$th term in this resolution is given by:
 \[P_d=\begin{cases}
 P(\l_{d+2})\oplus 2P(\l_{d})\oplus \cdots\oplus 2P(\l_2) \quad \text{if $d\geq 2$ is even,}\\
 P(\l_{d+2})\oplus 2P(\l_{d})\oplus \cdots\oplus 2P(\l_3)\oplus P(\l_1)\oplus P(\l_0)  \quad \text{if $d\geq 1$ is odd.}\\
 \end{cases}\]
 In the above, $2P(\l)$ means $P(\l)\oplus P(\l)$. Thus the number of direct summands in $P_d$ is either $d+1$ if $d\geq 2$ is even or it is $d+2$ if $d\geq 1$ is odd. This shows that $z_{\F}(L(\l_l))=2$.
 	
 The $d$th term in the minimal projective resolution of $L(\l_3)$ is given by:
\[P_d=\begin{cases}
 P(\l_{d+3})\oplus P(\l_{d+1})\quad \text{if $d=1$,}\\
 P(\l_{d+3})\oplus P(\l_{d+1})\oplus P(\l_1)\oplus P(\l_0) \quad \text{if $d=2$,}\\
 P(\l_{d+3})\oplus P(\l_{d+1})\oplus 2P(\l_{d-1})\oplus \cdots\oplus 2P(\l_2) \quad \text{if $d\geq 3$ is odd,}\\
 P(\l_{d+3})\oplus P(\l_{d+1})\oplus 2P(\l_{d-1})\oplus \cdots\oplus 2P(\l_3)\oplus P(\l_1)\oplus P(\l_0)  \quad \text{if $d\geq 4$ is even.}\\
 \end{cases}\]
 Thus the number of summands in $P_d$ is either $d+2$ if $d\geq 2$ is even or it is $d+1$ if $d\geq 1$ is odd. This shows that $z_{\F}(L(\l_3))=2$.
 	
 For $l\geq 3$, the projective cover of $L(\l_l)$ has the same structure as the projective cover of $L(\l_3)$. This gives  $z_{\F}(L(\l_l))=2.$
 \end{proof}

\begin{thm}
	Let $\fg=\dd,\,G(3),$ or $F(4)$. If $S$ is an atypical $\fg$-simple module then  $z_{\F}(S)=2$.
\end{thm}
\begin{proof}
Let $\fg=\dd$. The $z$-complexity is a categorical invariant, thus using the equivalence between the principal blocks of $\osp(3|2)$ and $\dd$ we have $z_{\F}(S)=2$ if $S$ is an atypical simple $\dd$-module in $\Ga_0$. For the simple modules in the other atypical blocks $\Ga_k$, we can compute the $z$-complexity by writing an explicit minimal projective resolution. By \cite[Theorem~3.1.1]{Germoni}, the projective cover $P(\l_{k,l})$ of the simple module $L(\l_{k,l})$ has the radical structure:
\[\begin{tikzpicture}

\node (n1) at (1,1) {$L(\l_{k,l})$};
\node (n2) at (0,0)  {$L(\l_{k,l-1})$};
\node (n3) at (2,0)  {$L(\l_{k,l+1})$};
\node (n4) at (1,-1) {$L(\l_{k,l})$};

\foreach \from/\to in {n1/n2,n1/n3,n2/n4,n3/n4}
\draw (\from) -- (\to);

\end{tikzpicture}\]
 The $d$th term in the minimal projective resolution of $L(\l_{k,l})$ is given by:
\[P_d=\begin{cases}
  P(\l_{k,l\pm d})\oplus P(\l_{k,l\pm (d-2)})\oplus \cdots\oplus P(\l_{k,l\pm 1}) \quad \text{if $d\geq 1$ is odd,}\\
  P(\l_{k,l\pm d})\oplus P(\l_{k,l\pm (d-2)})\oplus \cdots\oplus P(\l_{k,l})  \quad \text{if $d\geq 2$ is even.}\\
  \end{cases}\]
In the above, $P(\l_{k,l\pm i})$ means $P(\l_{k,l+i})\oplus P(\l_{k,l-i})$. Thus the number of direct summands in $P_d$ is $d+1$. This shows that $z_{\F}(L(\l_{k,l}))=2$.

Let $\fg=G(3)$. By \cite[Theorem~4.1.1]{Germoni}, every atypical block $\Ga_k$ is equivalent to the principal block of $\osp(3|2)$. The result follows in this case since the $z$-complexity is a categorical invariant. 

Let $\fg=F(4)$. The radical layer structure of the projective indecomposable modules over $F(4)$ is the same as that over $G(3)$. This will give the same projective resolutions over $F(4)$. Thus the result follows in this case.
\end{proof}

\subsection{Geometric interpretation of the $z$-complexity}
We give an explicit detecting subsuperalgebra that will be used to interpret the $z$-complexity geometrically. Let $\ff_{\1}\subseteq \fg_{\1}$ be the span of the root vectors $x_\alpha, x_{-\alpha}$ where
 \[\alpha=\begin{cases}
  \varepsilon_1+\delta \quad \text{if $\fg=\osp(3|2)$}\\
  \varepsilon_1+\varepsilon_2+\varepsilon_3\quad \text{if $\fg=\dd$}\\
  \varepsilon_3+\delta \quad \text{if $\fg=G(3)$}\\
  \frac{1}{2}(\e_1+\e_2+\e_3+\delta) \quad \text{if $\fg=F(4)$}.
 \end{cases}\]
 
   Set $\ff_{\0}=[\ff_{\1},\ff_{\1}]$. We define a three-dimensional subalgebra of $\fg$ by
$\ff:=\ff_{\0}\oplus \ff_{\1}$. The Lie superalgebra $\ff$ is classical and so has a support variety theory. Furthermore, as $[\ff_{\0},\ff_{\1}]=0$, it follows that these varieties admit a rank variety description and, in particular, can be identified as subvarieties of $\ff_{\1}$, i.e.,
\begin{align*}
\fv_{(\ff,\ff_{\0})}(M)=\fv_{\ff_{\1}}^{rank}(M)=\{y\in \ff_{\1}\mid \text{$M$ is not projective as $U(\langle y \rangle)$-module}\}\cup\{0\}.
\end{align*}
Using this detecting subsuperalgebra, we have the following geometric interpretation of the $z$-complexity:
\begin{thm}
If $S$ is a simple module over $\osp(3|2),\,\dd,\, G(3),$ or $F(4)$, then
$$z_{\F}(S)=\dim \fv_{(\ff,\ff_{\0})}(S).$$
\end{thm}
\begin{proof}
Note that $\fv_{(\ff,\ff_{\0})}(\C)=\ff_{\1}$. Thus $\dim \fv_{(\ff,\ff_{\0})}(\C)=z_{\F}(\C)=2$. Moreover, \cite[Theorem~4.1.1]{kujawa} implies that for any atypical simple module $S$, we have $\fv_{(\ff,\ff_{\0})}(S)=\fv_{(\ff,\ff_{\0})}(\C)$ since $S$ and $\C$ have the same atypicality. The result follows for atypical simple modules. If $S$ is typical, then it is projective, hence $\dim \fv_{(\ff,\ff_{\0})}(S)=z_{\F}(S)=0.$
\end{proof}

\end{document}